\documentclass[final]{siamltex}

\usepackage{amssymb}
\usepackage{amsmath}
\usepackage{epstopdf}
\usepackage{rotating}
\usepackage{wasysym}

\newtheorem{lem}{Lemma}[section]
\newtheorem{thm}{Theorem}[section]
\newtheorem{defd}{Definition}[section]
\newtheorem{prop}{Proposition}[section]
\newtheorem{cor}{Corollary}[section]
\newtheorem{remark}{Remark}[section]

\newcommand{\real}{{\mathbb{R}}}
\newcommand{\eref}[1]{(\ref{#1})}
\newcommand{\tK}{\tilde K}
\newcommand{\tKi}{\tilde K^{(i)}}
\newcommand{\Ki}{ K^{(i)}}
\newcommand{\tKj}{\tilde K^{(j)}}
\newcommand{\Kj}{ K^{(j)}}
\newcommand{\tP}{\tilde P}
\newcommand{\tPi}{\tilde P^{(i)}}
\newcommand{\Pii}{ P^{(i)}}
\newcommand{\ta}{\tilde a}
\newcommand{\txi}{\tilde \xi}
\newcommand{\txii}{\tilde \xi^{(i)}}
\newcommand{\xii}{ \xi^{(i)}}
\newcommand{\tomega}{\tilde \omega}

\title{ Implicit--Explicit  Runge-Kutta schemes for numerical discretization of
optimal control problems \thanks{  This work has been supported by DFG
HE5386/7-1,  HE5386/8-1 and by DAAD  50727872, 50756459 and
54365630. We also acknowledge the support of Ateneo Italo-Tedesco
(AIT) under the Vigoni project 2010-2012 "Adjoint implicit-
-explicit methods for the numerical solution to optimization
problems". }
}

\author{M. Herty\thanks{RWTH Aachen University, Templergraben 55, D-52065 Aachen, GERMANY.
{\tt \{herty,steffensen\}@mathc.rwth-aachen.de} }  \and  L.
Pareschi\thanks{University of Ferrara, Department of
Mathematics, Via Machiavelli 35, I-44121 Ferrara, ITALY. {\tt
lorenzo.pareschi@unife.it}} \and S. Steffensen\footnotemark[1] }

\begin{document}

\maketitle

\begin{abstract}
Implicit-explicit (IMEX) Runge-Kutta methods play a major rule in
the numerical treatment of differential systems governed by stiff
and non-stiff terms. This paper discusses order conditions and
symplecticity properties of a class of IMEX Runge--Kutta methods
in the context of optimal control problems. The analysis of the
schemes is based on the continuous optimality system. Using
suitable transformations of the adjoint equation, order conditions
up to order three  are proven as well as the relation between
adjoint schemes obtained through different transformations is
investigated. Conditions for the IMEX Runge--Kutta methods to be
symplectic are also derived. A numerical example illustrating the
theoretical properties is presented.
\end{abstract}

\begin{keywords}
 IMEX schemes, optimal control,
symplectic methods, Runge-Kutta methods
\end{keywords}

\begin{AMS}
65Kxx, 49M25, 65L06
\end{AMS}

\pagestyle{myheadings}
\thispagestyle{plain}
\markboth{}{}

\section{Introduction}\label{intro}

Recently, there has been intense research on the time
discretization of optimal control problems involving differential
equations. Such methods have found widespread applications in
aerospace and mechanical engineering, the life sciences, and many
other disciplines. In particular, properties of Runge--Kutta
methods have been investigated for example in
\cite{Hager00,BonnansLaurent-Varin06,Kaya10,LangVerwer2011,HertySchleper11,DH01,DH02}.
Hager \cite{Hager00} investigated order conditions (up to order
four) for Runge--Kutta methods applied to optimality systems. This
work has been later extended \cite{BonnansLaurent-Varin06, Kaya10} and
also properties of symplecticity of the scheme have been studied,
see also \cite{CHV08}. Further studies of discretizations of state
and control constrained problems using Runge--Kutta methods have
been  conducted in \cite{DH01,DH02}. The observations lead to the
idea to extend also other schemes  like W-methods to optimal
control problems \cite{LangVerwer2011}. Further, automatic
differentiation has been applied to Runge--Kutta discretizations
\cite{Walther2007}.

In many practical application involving systems of differential
equations of the form
\[
y'(t)=f(y(t),t)+g(y(t),t),
\]
where $f$ and $g$ are eventually obtained as suitable
finite-difference or finite-element approximations of spatial
derivatives, the time scales induced by the two operators may be
considerably different. Let us assume that $f$ is the non-stiff
term and $g$ the stiff one. Although the problem is stiff as a
whole, the use of fully implicit solvers originates a nonlinear
system of equations involving also the non-stiff term $f$ which
quite often represent the most expensive/difficult term in the
computation. Thus it is highly desirable to have a combination of
implicit and explicit (IMEX) discretization terms to resolve stiff
and non--stiff dynamics accordingly. For Runge-Kutta methods such
schemes have been studied in
\cite{ARS97,Bos07,BPR11,DP11,Hig2006,KC03,PR,PR03}. Among the
prominent examples are the numerical integration of hyperbolic
conservation laws, convection--diffusion equations and singular
perturbed problems.

As discussed in \cite{KC03, PR} the construction of such methods
implies new difficulties due to the appearance of coupled order
conditions and to the possible loss of accuracy close to stiff
regimes. The present work is concerned with the use of
implict--explicit methods in the context of optimal control
problems. Here we focus our attention to the order condition of
the adjoint IMEX system and its symplecticity property leaving to
further research specific application to partial differential
equations. We refer to \cite{HertyBanda11, HertySchleper11} for
examples of applications to hyperbolic problems.

The general IMEX Runge-Kutta scheme is introduced in Section 2 as
well as its discrete adjoint equations.  A transformation of these
equations is proposed in order to later on analyse order
conditions and symplecticity properties. Since the presented
transformation is different from the one used for example in
\cite{Hager00,HertySchleper11} we also discuss the
relation between the schemes obtained by using the different
transformations. The existing relations are summarized in Figure
\ref{fig-summary}. We furthermore investigate the relation
 between the two possible approaches  to derive the optimality system: we prove in Theorem 2.1
that  {discretize--then--optimize} and
{optimize--then--discretize} are equivalent. The order conditions
up to order three are summarized in Theorem 3.1 and the results on
symplecticity are given in Theorem 3.2. A numerical example is
presented in Section 4. Examples of IMEX Runge-Kutta schemes up to
order three are reported in a separate appendix.

\section{IMEX Runge-Kutta methods for optimal control problems}\label{sec2}

\subsection{The optimal control problem}
We consider optimal control problems for ordinary differential equations
of  type (\ref{OCP}):

\begin{subequations}\label{OCP}
\begin{align}
 (OCP) \qquad&         \min \;  j(y(T)) \quad \mbox{ such that }Ê\\
                         & \dot y(t)= f(y(t),u(t))+ g(y(t),u(t)),\qquad t\in [0,T]\\
               &      y(0) =  y^0 .
  \end{align}
\end{subequations}

Related to the optimal control problem we introduce the
Hamiltonian function $H$ as $H(y,u,p):= p ^T( f(y,u)+ g(y,u)).$
Under  appropriate conditions it is well--known  \cite{Hestenes80, Troutman1996}
that the first--order optimality conditions for (\ref{OCP})
are
\begin{subequations}\label{osOCP}
\begin{eqnarray}
 \dot y&=& ~ ~ \,H_p(y,u,p)=  f(y,u) +  g(y,u), \qquad \qquad ~ ~ y(0)\,=\, y^0\label{osOCPa}\\
 \dot p&=& - H_y(y,u,p)= - f_y(y,u)^T p -   g_y(y,u)^T  p, \qquad p(T)\,=\, j'(y(T))\label{osOCPb}\\
0&=& ~~H_u(y,u,p)=f_u(y,u)^T p +  g_u(y,u)^T  p.\label{osOCPc}
\end{eqnarray}
\end{subequations}
The equation \eref{osOCPa} is called state equation and
\eref{osOCPb} is called adjoint equation. We are interested in
implicit--explicit  Runge--Kutta  (IMEX--RK) discretizations for
\eref{osOCPa} and \eref{osOCPb}.
To be more precise,  we treat $f$ by an explicit method and assume
that $g$ enjoys some stiffness so that an implicit method is
required. Therefore, the discretization of the  general case of
\eref{osOCPa}  leads to  two different schemes for $f$ and $g$,
respectively. A corresponding Runge-Kutta discretization scheme
\cite{PR03,PR} with $s$ stages is given by
\begin{subequations}\label{SRK-1}
 \begin{align}
Y^{(i)}_n&  =&y_n + h\sum_{j=1}^{s} \ta_{ij}f(Y_n^{(j)},u_n^j) + h\sum_{j=1}^s
a_{ij}g(Y_n^{(j)},u_n^j)\qquad i=1,..,s\\
y_{n+1}&=& y_n +
h\sum_{i=1}^s \tomega_if(Y_n^{(i)},u_n^i) + h\sum_{i=1}^s \omega_ig(Y_n^{(i)},u_n^i),\qquad n=0,1,2,.
 \end{align}
\end{subequations}
where $A,\tilde A,\omega,\tomega$ are the associated Runge--Kutta coefficient matrices
 and the Runge--Kutta weights, respectively.
We refer to methods where an explicit scheme for $f$ and an
implicit scheme for $g$ is used as as IMEX--RK schemes.  Their
properties have been investigated for example in
\cite{ARS97,PR03, PR}. A particularly interesting
subclass is the class of diagonally implicit IMEX--RK methods.
\begin{defd}
We call the method \eqref{SRK-1} diagonally implicit  IMEX--RK
method, iff $\ta_{ij}=0$ for $j\geq i$ and $a_{ij}=0$ for all
$j>i$.
\end{defd}

In order to simplify the notation in the sequel we do not
truncate the corresponding sums and, if not stated otherwise, all following results are given for a general implicit--explicit methods. Only later
we will reduce the
investigation to the class of diagonally implicit methods.
\par
As in \cite{BonnansLaurent-Varin06,PR,PR03} we use an equivalent formulation of the IMEX--RK
scheme in order to derive the discrete first--order optimality conditions. Instead of representation
 \eref{SRK-1} we use the equivalent formulation (\ref{SRK-2})
\begin{subequations}\label{SRK-2}
 \begin{eqnarray}
\tKi_n&  =& f\left(y_n+h\sum_{j=1}^s \ta_{ij}\tKj + h\sum_{j=1}^s a_{ij}\Kj_n,u_{n}^i\right)\\
\Ki_n&  =& g\left(y_n+h\sum_{j=1}^s \ta_{ij}\tKj + h\sum_{j=1}^s a_{ij}\Kj_n,u_{n}^i\right)\\
y_{n+1}&=& y_n +h\sum_{i=1}^s \tomega_i\tKi_n + h\sum_{i=1}^s \omega_i\Ki_n.
 \end{eqnarray}
\end{subequations}
Note that due to the two different schemes for $f$ and $g$, respectively,
 we introduce two auxiliary variables $\tKi$ and $\Ki$. These lead to additional discrete adjoint equations
compared with the formulation in \cite{BonnansLaurent-Varin06}.  The associated discretized optimal
control problem to (\ref{OCP}) using IMEX--RK is hence given by
\begin{subequations}\label{DOP}
\begin{align}
 (DOP) \quad            & \min  j(y_N)  \mbox{ such that }Ê\\
               			& \tKi =  f\left(y_n+h\sum_{j=1}^s \ta_{ij}\tKj + h\sum_{j=1}^s a_{ij}\Kj, u_n^i\right)\\
&\Ki = g\left(y_n+h\sum_{j=1}^s \ta_{ij}\tKj + h\sum_{j=1}^s a_{ij}\Kj,u_n^i\right)\\
&y_{n+1} =  y_n +h\sum_{i=1}^s \tomega_i\tKi + h\sum_{i=1}^s \omega_i\Ki, \quad
 y_0 = y^0.
  \end{align}
\end{subequations}
Clearly, the Lagrangian  is
\begin{eqnarray}\label{lagrangian}
 \mathcal L(y, K,\tK,p,\txi,\xi)&=& j(y_N)+p_0^T(y_0-y^0)\nonumber\\
&&\quad +\sum_{n=0}^{N-1}
\left[ p_{n+1} ^T\left(-y_{n+1}+y_n +h\sum_{i=1}^s \tomega_i\tKi_n +
 h\sum_{i=1}^s \omega_i\Ki_n\right)\right.\\[0.5em]
&&\left.\qquad  + \sum_{i=1}^s (\txii_n)^T(-\tKi_n+ f(Y^{(i)}_n,u_n^i))+
\sum_{i=1}^s (\xii_n)^T(-\Ki_n+ g(Y^{(i)}_n,u_n^i))\right]\nonumber
\end{eqnarray}
where $
 Y^{(i)}_n:=y_n+h\sum_{j=1}^s \ta_{ij}\tKj_n + h\sum_{j=1}^s a_{ij}\Kj_n .
$
 Here, the vectors $\txi,\xi$ and $p$ are the Lagrange
multipliers  corresponding
  to the equality constraints given by the initial condition and system \eref{SRK-2}, respectively.
   For the  first order necessary optimality conditions for \eref{DOP}
we obtain the feasibility conditions given by \eref{SRK-2} and furthermore the discrete adjoint
equations which are  derived  upon  differentiation of $ \mathcal L(y,p,\txi,\xi)$ with respect to $\Ki_n,\tKi_n$ and $y_n$,
respectively. The system of adjoint equations reads
\begin{subequations}\label{ARK-1}
 \begin{eqnarray}
\txii_n&  =&h\, \tomega_i \,p_{n+1}+h\sum_{j=1}^s \ta_{ji}\,f_y(Y^{(j)}_n,u_n^j)^T\txi^{(j)}_n
+ h\sum_{j=1}^s \ta_{ji}\,g_y(Y^{(j)}_n,u_n^j)^T \xi^{(j)}_n\\
\xii_n&  =&h\, \omega_i \,p_{n+1}+h\sum_{j=1}^s a_{ji}\,f_y(Y^{(j)}_n,u_n^j)^T \txi^{(j)}_n
+ h\sum_{j=1}^s a_{ji}\,g_y(Y^{(j)}_n,u_n^j)^T\xi^{(j)}_n\\
p_{n}&=& p_{n+1} +\sum_{i=1}^s  \,f_y(Y^{(i)}_n,u_n^i)^T \txii _n
+\sum_{i=1}^s  \,g_y(Y^{(i)}_n,u_n^i)^T\xii_n, \quad
p_N = j'(y_N),
 \end{eqnarray}
\end{subequations}
where the index range for $n$ is $N-1,..,0$ and the intermediate adjoint states have to be  computed for $i=1,..,s$.
 The discretization method \eref{ARK-1} is not yet in a standard RK notation.
Similar to \cite{ BonnansLaurent-Varin06,Hager00,
HertySchleper11} we have the following result.
\begin{prop}\label{ARK-reformulation}
 If we assume that $\tomega_i\neq0$ and $\omega_i\neq0$ then \eref{ARK-1} can be rewritten as
\begin{subequations}\label{ARK-2}
 \begin{eqnarray}
\tPi&  =&p_n - h\sum_{j=1}^s \tilde \alpha_{ij}\,f_y(Y^{(j)}_n,u_n^j)^T \tP^{(j)}
- h\sum_{j=1}^s \alpha_{ij}\,g_y(Y^{(j)}_n,u_n^j)^T P^{(j)}\label{ARK-2a}\\
\Pii&  =&p_n-h\sum_{j=1}^s \tilde \beta_{ij}\,f_y(Y^{(j)}_n,u_n^j)^T\tP^{(j)}
- h\sum_{j=1}^s \beta_{ij}\,g_y(Y^{(j)}_n,u_n^j)^T P^{(j)} \label{ARK-2b}\\
p_{n+1}&=& p_{n} - h\sum_{i=1}^s \tomega_i \,f_y(Y^{(i)}_n,u_n^i)^T \tPi
-h \sum_{i=1}^s \omega_i \,g_y(Y^{(i)}_n,u_n^i)^T \Pii, \label{ARK-2c}
 \end{eqnarray}
\end{subequations}
where the coefficients  $\tilde \alpha_{ij}, \alpha_{ij}, \tilde \beta_{ij}$
and $\beta_{ij}$  are given by
\begin{align*}
 \tilde \alpha_{ij}:= \tomega_j-\frac{\tomega_j}{\tomega_i} \ta_{ji},  \quad
\alpha_{ij}:= \omega_j-\frac{\omega_j}{\tomega_i} \ta_{ji}, \quad
\tilde \beta_{ij}:= \tomega_j-\frac{\tomega_j}{\omega_i} a_{ji}, \quad
\beta_{ij}:= \omega_j-\frac{\omega_j}{\omega_i} a_{ji} .
\end{align*}
\end{prop}
\begin{proof}
If $\tomega_i\neq0$ and $\omega_i\neq0$, then we can define new variables
\begin{equation}\label{var-trans}
  \tPi_n:= \frac{\txii_n}{h \,\tomega_i} \qquad \mbox{and} \qquad  \Pii_n:= \frac{\xii_n}{h \,\omega_i}
\qquad (i=1,..,s; \quad n=0,..,N-1)\,.
\end{equation}
We obtain  \eref{ARK-2}  using the definition of $\Pii_n$ and $ \tPi_n$
in  $\xii_n$ and $\txii_n$, respectively.
\end{proof}

\vspace{1em}
\begin{remark}
Referring to the classification of IMEX--RK methods given in \cite{BPR11}, Proposition
\ref{ARK-reformulation} is extended to IMEX schemes of type ARS with $\omega_1=0$.
In this case, define $\Pii_n$ for $i=2,..,s$ and $\tPi_n$ for $i=1,..,s$ as in \eref{var-trans} and
use the transformation to obtain \eref{ARK-2a} and \eref{ARK-2b} for $i=1,..,s$ and $i=2,..,s$,
respectively, and for the further equation we set
\[P^{(1)}_n: =p_{n} - h\sum_{i=1}^s \tomega_i \,f_y(Y^{(i)}_n,u_n^i)^T \tPi
-h \sum_{i=1}^s \omega_i \,g_y(Y^{(i)}_n,u_n^i)^T \Pii\]
i.e. we use again \eref{ARK-2b} and define the coefficients
$\tilde \beta_{1j}:=\tomega_j$ and $\beta_{1j}:=\omega_j$.
The remaining coefficients are defined as in Proposition \ref{ARK-reformulation}.
\end{remark}

\subsection{Discrete and continuous optimality systems}

We prove  the following results on the relations
depicited in Figure \ref{fig-summary}. The
 discrete optimality system of \eref{DOP} represents
a RK discretization of the continuous optimality system \eref{osOCP}.
The system obtained by discretizing the continuous optimality system \eref{OCP}
and by optimizing the discretized optimal control problem coincide.

\begin{figure}[t]
\begin{center}
\includegraphics[height=8cm,width=12cm]{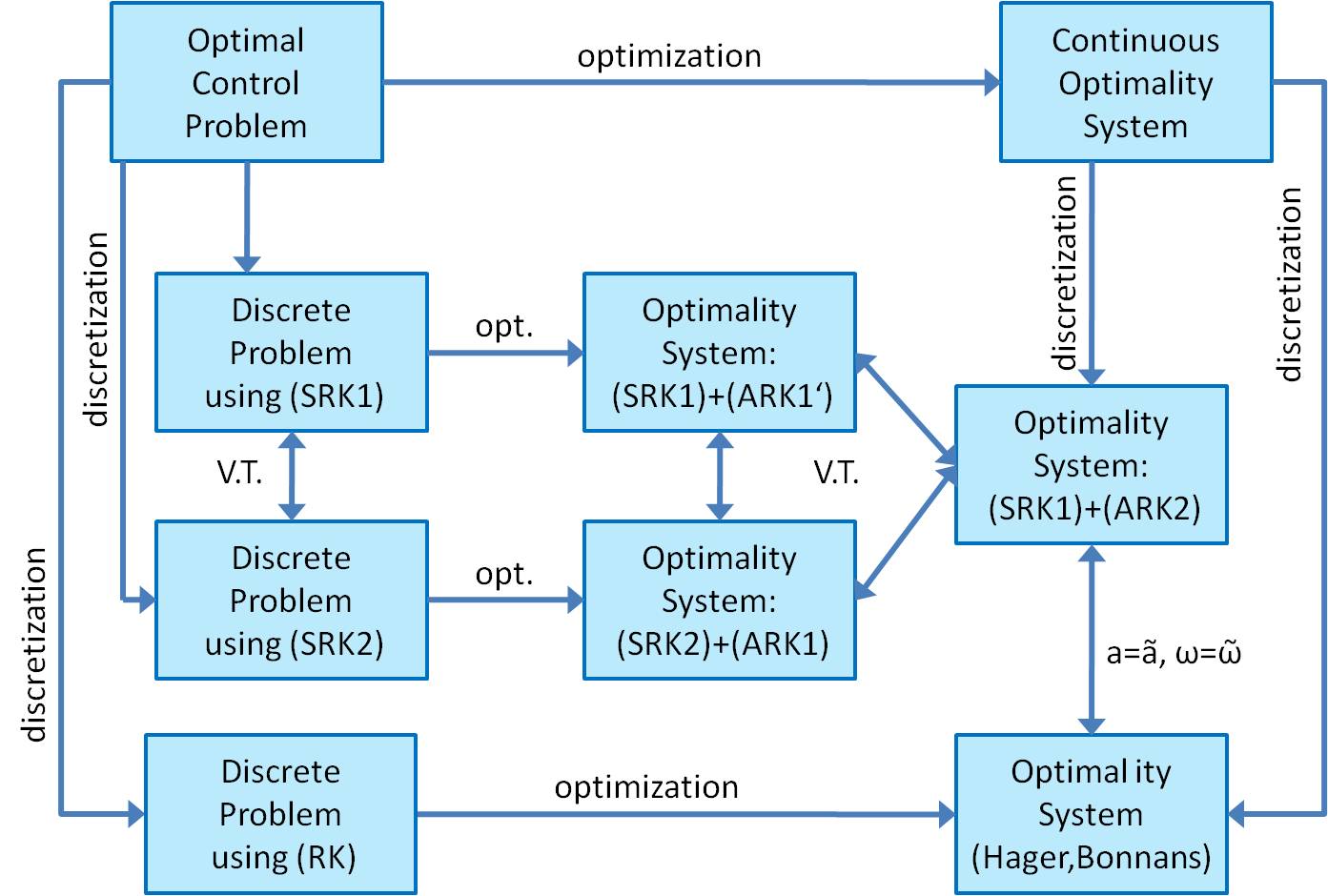}
\caption{Diagram summarizing the relations discussed. Here
``V.T.'' stands for ``variable transformation'',
(RK) denotes the single RK method for \eref{osOCPa}  as in \cite{BonnansLaurent-Varin06,Hager00},
(SRK1)  denotes system \eref{SRK-1}, (SRK2)
 denotes system \eref{SRK-2}, (ARK1) denotes system \eref{ARK-1},
(ARK2) denotes system \eref{ARK-2} and
(ARK1') denotes the discrete adjoint system
\eref{ARK-HS}.
\label{fig-summary}}
\end{center}
\end{figure}

\begin{thm}\label{discreteOS}
We consider the first--order necessary optimality conditions \eref{osOCP} of the optimal control problem \eref{OCP}.
\par
The equations \eref{SRK-1} and  \eref{ARK-2} are the discrete state and adjoint equations of the optimality system
to the problem
$ \min j(y_N) \mbox{ subject to }\eref{SRK-1}.$
\par
The equations \eref{SRK-1}  and \eref{ARK-2} are a discretization of the continuous state \eref{osOCPa} and
adjoint \eref{osOCPb} equation.
\end{thm}

The proof of Theorem \ref{discreteOS}  follows showing that
system  \eref{SRK-1} yields
a discretization scheme \eref{ARK-HS} for \eref{osOCPb} and the latter is  transformed into
\eref{ARK-2} by a   variable transformation of the  intermediate states.

\begin{lem}\label{Lem1}
Given the optimal control problem \eref{OCP} and the RK discretization \eref{SRK-1} for
the state equation \eref{osOCPa}. Then, the associated RK discretization of the discrete adjoint
equation is equivalent to \eref{ARK-2}.
\end{lem}

\begin{proof}
The associated discretized optimal control problem is given by
\begin{subequations}\label{DOP1}
  \begin{align}
 (DOP_1) \quad           & \min  j(y_N) \\
 &Y^{(i)}_n = y_n + h\sum_{j=1}^{s} \ta_{ij}f(Y_n^{(j)},u_n^j) + h\sum_{j=1}^s
a_{ij}g(Y_n^{(j)},u_n^j),\\
& y_{n+1} = y_n +h\sum_{i=1}^s \tomega_if(Y_n^{(i)},u_n^i) +
h\sum_{i=1}^s \omega_ig(Y_n^{(i)},u_n^i), \quad y_0 = y^0.
  \end{align}
\end{subequations}
The stationary points of the Lagrangian yield  for $i=1,\dots,s$ and $n=1, \dots, N-1$
\begin{subequations}\label{ARK-HS}
 \begin{eqnarray}
\zeta^{(i)}_n&  =&h\left(\tomega_i f_y(Y_n^{(i)},u_n^i)+ \omega_i g_y(Y_n^{(i)},u_n^i)\right)^Tp_{n+1} \label{ARK-HSa} \\
&&\qquad + \sum_{j=1}^{s} \ta_{ji}\,f_y(Y_n^{(i)},u_n^i) ^T\zeta^{(j)} +  h\sum_{j=1}^s
a_{ji}\,g_y(Y_n^{(i)},u_n^i)^T \zeta^{(j)}  \\
p_{n}&=& p_{n+1} + \sum_{i=1}^s\zeta^{(i)}_n, \quad   p_N=j'(y_N)\label{ARK-HSb}\\
0&=& h\left(\tomega_i f_u(Y_n^{(i)},u_n^i)+ \omega_i g_u(Y_n^{(i)},u_n^i)\right)^Tp_{n+1}
 \label{ARK-HSc}\\
&&\qquad + \sum_{j=1}^{s} \ta_{ji}\,f_u(Y_n^{(i)},u_n^i) ^T\zeta^{(j)} + h\sum_{j=1}^s
a_{ji}\,g_u(Y_n^{(i)},u_n^i)^T \zeta^{(j)}
 \end{eqnarray}
\end{subequations}
The system \eref{ARK-HS}
is not in a standard RK formulation. As in \cite{Hager00,HertySchleper11}
 we  reformulate this system
as a standard partitioned RK method for $p$.
In contrast to \cite{Hager00}  we introduce two new variables:
\begin{equation}\label{vartrans}
 \tPi_n:= p_{n+1}+\sum_{j=1}^s \frac{\ta_{ji}}{\tomega_i}\zeta^{(j)}_n\qquad \mbox{and} \qquad
\Pii_n:=p_{n+1}+\sum_{j=1}^s \frac{a_{ji}}{\omega_i}\zeta^{(j)}_n,
\end{equation}
for all $i=1,..,s$ and $n=0,..,N-1$. Then, \eref{ARK-HS},
yields the equivalent formulation of  \eref{ARK-2} with coefficients
\begin{align*}
 \check\alpha_{ij}:= \frac{\tomega_j}{\tomega_i} \ta_{ji}, \quad
\hat\alpha_{ij}:=\frac{\omega_j}{\tomega_i} \ta_{ji}, \quad
\check \beta_{ij}:= \frac{\tomega_j}{\omega_i} a_{ji},  \quad
\hat\beta_{ij}:=\frac{\omega_j}{\omega_i} a_{ji} .
\end{align*}
Rewriting the resulting backward scheme as a forward scheme then yields \eref{ARK-2}.
\end{proof}

\begin{lem}\label{Lem2}
Assume  the optimal control problem \eref{OCP} and the RK discretization
scheme \eref{SRK-1} for the state equation \eref{osOCPa} are given.
The discrete state and adjoint equation  \eref{SRK-1} and \eref{ARK-2}, respectively, are
equivalent to a  RK discretization
of the continuous optimality system \eref{osOCP}.
\end{lem}

\begin{proof}
 Note that adding a corresponding set of equalities for $\dot{\tilde p}$, the  system \eref{SRK-1}
together with \eref{ARK-2} corresponds to a standard additive RK scheme for the extended (modified)
 system with $\tilde p(T)= p(T)$ and
\begin{subequations}\label{aux-hamiltonian}
 \begin{eqnarray}
 \dot{ y}&=&f(y,u) + g(y,u)\label{aux-hamiltonian-a}\\
 \dot{\tilde p}&=&f_y(y,u)^T\tilde p + g_y(y,u)^Tp\label{aux-hamiltonian-b}\\
 \dot p&=&f_y(y,u)^T \tilde p + g_y(y,u)^Tp\label{aux-hamiltonian-c}
\end{eqnarray}
\end{subequations}
The system \eref{SRK-1} and \eref{ARK-2} form an additive RK method for \eref{aux-hamiltonian}.
Since $\tilde p$ and $p$ have the same initial conditions, we have
 $\tilde p \equiv p$. Therefore the solution to   \eref{aux-hamiltonian-b} and \eref{aux-hamiltonian-c}
 are equivalent to the solution to \eref{osOCPb}.
 Finally,  the approximate adjoints $p_h$ and $\tilde p_h$
 are the discretized solution to  \eref{osOCPb} for the same initial equations.
\end{proof}

 Theorem \ref{discreteOS} follows directly by Lemma \ref{Lem1} and Lemma \ref{Lem2}.
We have a few comments on the implications of the previous theorem.
\par
In view of \eref{ARK-2},  we obtain four additional RK coefficient matrices and therefore four RK  methods.
Furthemore, we have the two  RK methods for the IMEX discretization  \eref{SRK-1}.
If we additionally assume that $\omega_i = \tomega_i$  holds for all $i=1,..,s$, then we
obtain $\tilde \alpha_{ij}= \alpha_{ij}$
and $\tilde \beta_{ij}= \beta_{ij}$ for all $i,j = 1,..,s$. Hence we obtain only two additional RK schemes for the
optimality system \eref{osOCP}.
Moreover, the equation is independent of the discretization of \eref{osOCPa}.
This implies that either $(DOP)$ or $(DOP_1)$ can be used to discretize the problem.
\par
Also, if the Runge--Kutta  methods for $f$ and $g$ coincide, i.e. $\ta_{ij}=a_{ij}$ and
$\tomega_i=\omega_i$ for all $i,j$, then the discrete adjoint scheme
\eref{ARK-2} coincides with the one derived by Hager \cite{Hager00} and it also
coincides with  Bonnans
et al \cite{BonnansLaurent-Varin06}, respectively. We therefore recover their results.
If $g\equiv 0$, then an explicit RK method for $f$ in \eref{SRK-1} yields
a ``backward" explicit method for \eref{ARK-1} and an implicit method for
\eref{ARK-2}. Here, the variables $\xi_n^{(i)}$ and  $P_n^{(i)}$ vanish since there is
no contribution to  $p_{n}$ and $p_{n+1}$, respectively.
 If $f\equiv0$, then the backward variant of \eref{ARK-2}
 yields a``backward" diagonally implicit method for $g$.
\par
 The system \eref{ARK-1} is a  ``backward in time`` scheme for the adjoint variable $p$, with initial value
$j'(y_N)$. Note that if \eref{SRK-1} is  an implicit-explicit (IMEX) RK scheme with a diagonally
 implicit method (DIRK) for $g$, then \eref{ARK-1}  also is an IMEX method with
a diagonally implicit method for the terms that belong to $g$. Hence, the presentation
simplifies in this case which also has been discussed in \cite{HertySchleper11}.
 Although  \eref{ARK-2} is more suitable for  theoretical investigations,
 for an efficient implementation of the scheme we used formulation \eref{ARK-1}.
\par
 The discretization of equation  \eref{osOCPc}
 is straight--forward
\begin{equation}\label{Hu}
\tilde \omega_k  f_u(Y_n^{(k)},u_n^k)^T\tilde P^{(k)}_n + \omega_k  g_u(Y_n^{(k)},u_n^k)^TP^{(k)}_n=0.
\end{equation}
Next,  we prove that for a suitable
discretization $H^h(y_n,p_{n+1},u_n)$ of the Hamiltonian
$H(y,p,u)$  \eref{Hu} is equal to $\nabla_{u_k}H^h(y_n,p_{n+1},u_n)=0$.
Hence, Lemma  \ref{Lem-Hu} shows that \eref{Hu} is a valid discretization of
$H_u(y,p,u)=0$ in the limit $h\rightarrow 0$.

\begin{lem}\label{Lem-Hu}
 Let
$$
H^h(y_n,p_{n+1},u_n):=p_{n+1}^T\left[\sum_{i=1}^s (\tomega_i f^{(i)}+\omega_i g^{(i)} )\right]
 $$
with $f^{(i)}:=f(Y_n^{(i)},u_n^i)$ and $g^{(i)}:=g(Y_n^{(i)},u_n^i)$. Then,
\[
 \nabla_{u_k}H^h(y_n,p_{n+1},u_n)=
\tilde \omega_k  (f_u^{(k)})^T \tilde P^{(k)}_n + \omega_k (g_u^{(k)})^T P^{(k)}_n
\]
with $f_u^{(k)}:=f_u(Y_n^{(k)},u_n^k)$ and $g_u^{(k)}:=g_u(Y_n^{(k)},u_n^k)$.
\end{lem}
\begin{proof}
 In order to simplify the notation, we introduce the following matrices
\begin{align*}
 B&: s\times s \mbox{ block matrix with block entries } (i,j) : [\tilde a_{ji}(f_y^{(i)})^T+ a_{ji}(g_y^{(i)})^T],\qquad i,j=1,..,s\\
 C&: s\times 1 \mbox{ block matrix with block entries } \qquad  : [\tilde \omega_i (f_y^{(i)})^T+ \omega_i (g_y^{(i)})^T],\qquad i=1,..,s\\
 \tilde D_k&: 1\times s \mbox{ block matrix with block entries } \qquad  : [\tilde a_{ik}\mathbf{Id}], \qquad i=1,..,s\\
D_k&: 1\times s \mbox{ block matrix with block entries }  \qquad : [ a_{ik}\mathbf{Id}], \qquad i=1,..,s
\end{align*}
with $f_y^{(i)}:=f_y(Y_n^{(i)},u_n^i)$ and $g_y^{(i)}:=g_y(Y_n^{(i)},u_n^i)$.
Moreover, define $M:=(\mathbf{Id}-hB)$ and note that $M$ is invertible if $h$ is sufficiently small.
Using the just defined matrices, we  rewrite \eref{ARK-HSa} as
\begin{equation}\label{Lem-Hu1}
 M\zeta_n=h C p_{n+1}
\end{equation}
where $\zeta_n=((\zeta_n^{(1)})^T,...,(\zeta_n^{(s)})^T)^T$ and the equations of \eref{vartrans}
are
\begin{equation}\label{Lem-Hu2}
 \tilde \omega_k \tilde P_n^{(k)}=\tilde \omega_kp_{n+1}+\tilde D_k\zeta_n\qquad \mbox{and}\qquad
\omega_k P_n^{(k)}= \omega_kp_{n+1}+ D_k\zeta_n.
\end{equation}
Furthermore, differentiating \eref{SRK-1} with respect to $u_k$ yields
\begin{equation}\label{Lem-Hu3}
 \nabla_{u_k}Y_nM=h (f_u^{(k)})^T\tilde D_k+h (g_u^{(k)})^T D_k,
\end{equation}
where $Y_n=((Y_n^{(1)})^T,...,(Y_n^{(s)})^T)^T$. Now,  we  use \eref{Lem-Hu1}, \eref{Lem-Hu2}
and \eref{Lem-Hu3} and evaluate the gradient of $H^h(y_n,p_{n+1},u_n)$ with respect to $u_k$
\begin{eqnarray*}
 \nabla_{u_k}H^h(y_n,p_{n+1},u_n) = \\
\sum_{i=1}^s \left(\tilde \omega_i \nabla_{u_k}Y_n^{(i)}(f_y^{(i)})^Tp_{n+1} +
\omega_i \nabla_{u_k}Y_n^{(i)}(g_y^{(i)})^Tp_{n+1}\right)
 + \;\tilde \omega_k (f_u^{(k)})^Tp_{n+1} +  \omega_k (g_u^{(k)})^Tp_{n+1} \\
=  \nabla_{u_k}Y_n Cp_{n+1} + \tilde \omega_k (f_u^{(k)})^Tp_{n+1} +  \omega_k (g_u^{(k)})^Tp_{n+1} \\
=  \left(h (f_u^{(k)})^T\tilde D_k+h (g_u^{(k)})^T D_k\right)M^{-1}Cp_{n+1}
+ \;\tilde \omega_k (f_u^{(k)})^Tp_{n+1} +  \omega_k (g_u^{(k)})^Tp_{n+1} \\
=  (f_u^{(k)})^T(\tilde D_k\zeta_n+\tilde \omega_k p_{n+1}) +
(g_u^{(k)})^T(D_k\zeta_n+ \omega_k p_{n+1})\\
= \tilde \omega_k  (f_u^{(k)})^T \tilde P^{(k)}_n + \omega_k (g_u^{(k)})^T P^{(k)}_n.
\end{eqnarray*}
\hfill
\end{proof}

\section{ Properties of the IMEX Runge-Kutta discretizations}
Theoretical properties of the derived RK method for
system \eref{osOCP} corresponding
to the partitioned RK method
are considered.

\subsection{Order conditions}
We analyse the order conditions  for the RK method \eref{SRK-1}
together with \eref{ARK-2}. As in the proof of Lemma \ref{Lem2} we add
 the additional equation
\begin{equation}\label{tilde p}
 \tilde p_{n+1}= \tilde p_{n} - h\sum_{i=1}^s \tomega_i \,f_y(Y^{(i)}_n,u_n^i)^T \tPi
-h \sum_{i=1}^s \omega_i \,g_y(Y^{(i)}_n,u_n^i)^T \Pii,
\end{equation}
for $\tilde p$ such that the resulting method corresponds to a standard additive RK scheme for the
auxiliary problem \eref{aux-hamiltonian}. Since the system is completely
coupled,  we  also  obtain a similar coupling in the order conditions \cite{KC03}.
 \par
We  start with the analysis of first and second order conditions in the general case,
i.e. $A\neq \tilde A$ and $\omega\neq \tomega$.  We then restrict ourselves to the case
  $\tomega=\omega.$ In this case \textbf{all} additional coupling conditions for first and second
order are directly satisfied by the order condition for the forward IMEX scheme \eref{SRK-1}.
For third order we obtain  an additional condition.  If we consider the adjoint equation alone
those conditions have been studied in \cite{HertySchleper11} for a decoupled system
\eref{SRK-1}, \eref{ARK-2}.

 Define the coefficients
\begin{align*}
 c_i :=\sum_{j=1}^s a_{ij}, \;\;
\tilde c_i :=\sum_{j=1}^s \ta_{ij}, \;\;
\gamma_i  := \sum_{j=1}^s \alpha_{ij}, \;\;
 \tilde \gamma_i := \sum_{j=1}^s \tilde\alpha_{ij}, \;\;
\delta_i  := \sum_{j=1}^s \beta_{ij}, \;\;
 \tilde \delta_i := \sum_{j=1}^s \tilde\beta_{ij},
\end{align*}
and
\begin{eqnarray*}
d_j= \sum_{i=1}^s  \omega_i \tilde a_{ij}, \quad
\tilde d_j= \sum_{i=1}^s  \tomega_i \tilde a_{ij}, \quad
e_j= \sum_{i=1}^s  \omega_i  a_{ij} \quad  \mbox{and} \quad
\tilde e_j= \sum_{i=1}^s  \tomega_i a_{ij}.
\end{eqnarray*}

\begin{prop}\label{prop-order}
 Consider the additive Runge-Kutta method  \eref{SRK-1}
together with \eref{ARK-2} and \eref{tilde p} as a discretization scheme for \eref{aux-hamiltonian}.
For \eref{SRK-1} consider a diagonally implicit IMEX--RK method.
Then the following results hold true.
\begin{enumerate}
 \item The additive method is of first order, if the diagonally implicit IMEX--RK method \eref{SRK-1} is of first order.
\item The additive method is of second order, if the diagonally implicit IMEX--RK method \eref{SRK-1} is of second
order and  it additionally satisfies the following coupling conditions
\begin{equation}\label{2ordercond}
 \sum_{i=1}^s \frac{\omega_i}{\tomega_i}\, d_i =\frac{1}{2},\qquad
 \sum_{i=1}^s \frac{\omega_i}{\tomega_i}\, \tilde d_i =\frac{1}{2},\qquad
 \sum_{i=1}^s \frac{\tomega_i}{\omega_i}\, e_i =\frac{1}{2},\qquad
 \sum_{i=1}^s \frac{\tomega_i}{\omega_i}\,\tilde e _i =\frac{1}{2}.
\end{equation}
\end{enumerate}
\end{prop}
\begin{proof}
The first part of the theorem is trivial. For the second part we prove
\[
\qquad  \sum_{i=1}^s \omega_i \, \gamma_i =\frac{1}{2},
\qquad   \sum_{i=1}^s \omega_i \, \delta_i =\frac{1}{2}.
\]
By the second order of the IMEX--RK method \eref{SRK-1} and the
definition of $\gamma_i$ we have
\[
\sum_{i=1}^s \tomega_i \,\gamma_i =\sum_{i=1}^s \tomega_i - \sum_{i=1}^s
\sum_{j=1}^s \omega_j a_{ij} = 1- \sum_{i=1}^s w_j\, \tilde c_j = \frac{1}{2}.
\]
In the same way we  prove that
\[\sum_{i=1}^s \tomega_i \,\tilde \gamma_i
=\frac{1}{2},\qquad \sum_{i=1}^s \omega_i \, \delta_i =\frac{1}{2} \qquad \mbox{and} \qquad
\sum_{i=1}^s \omega_i \, \tilde \delta_i =\frac{1}{2}
\]
hold, if the second order conditions for \eref{SRK-1} are satisfied. However, since the remaining
coupling conditions cannot be simplified in the same way, we need to impose further the
 conditions \eref{2ordercond}. By \eref{2ordercond} we get
\[
  \sum_{i=1}^s \omega_i \, \gamma_i = 1 -  \sum_{i=1}^s \omega_i
\left( \sum_{j=1}^s\frac{\omega_j}{\tomega_i}\tilde a_{ji}\right)=
1-\sum_{i=1}^s \frac{\omega_i}{\tomega_i}\, d_i =\frac{1}{2}.
\]
Accordingly, the remaining conditions hold true
\[\sum_{i=1}^s \omega_i \,\tilde \gamma_i
=\frac{1}{2},\qquad \sum_{i=1}^s \tomega_i \, \delta_i =\frac{1}{2} \qquad \mbox{and} \qquad
\sum_{i=1}^s \tomega_i \, \tilde \delta_i =\frac{1}{2}.
\]
\hfill\end{proof}

The order conditions apply to all variables $y,p$ and $\tilde p$ and therefore
 $y$ and $p$ satisfy in particular equation \eref{osOCP}.
 Moreover, if $\omega_i=\tomega_i$ for all $i=1,..,s$ then the additional
second order conditions \eref{2ordercond} are  satisfied.
\begin{cor}\label{cor-order}
If $\omega_i=\tomega_i$ for all $i=1,..,s$ and the diagonally implicit IMEX--RK method \eref{SRK-1}
 is of second order, then the additive RK method  \eref{SRK-1}
together with \eref{ARK-2} and \eref{tilde p} is of second order.
\end{cor}
\begin{proof}
Since,
\begin{equation}
 \sum_{j=1}^s d_j=\sum_{i,j=1}^s \omega_j\tilde a_{ji}= \sum_{j=1}^s\omega_j
\sum_{i=1}^s \tilde a_{ji}=  \sum_{j=1}^s\omega_j\tilde c_j =\frac{1}{2}
\end{equation}
and in the same way
\[
  \sum_{j=1}^s \tilde d_j=\frac{1}{2},\qquad
 \sum_{j=1}^s e_j=\frac{1}{2}\qquad  \mbox{and}\qquad
 \sum_{j=1}^s \tilde e_j=\frac{1}{2},
\]
the conditions of \eref{2ordercond} hold.
\end{proof}

\begin{thm}\label{thm-3order}
If $\omega_i=\tomega_i$ for all $i=1,..,s$ and the diagonally implicit IMEX--RK method \eref{SRK-1}
 is of third order, then the additive RK method  \eref{SRK-1}
together with \eref{ARK-2} and \eref{tilde p} is of third order, provided that
\begin{equation}\label{3ordercond}
\sum_{i=1}^s \frac{d_i ^2}{\omega_i}=\frac{1}{3},\qquad
\sum_{i=1}^s \frac{e_i ^2}{\omega_i}=\frac{1}{3}\qquad \mbox{and}\qquad
\sum_{i=1}^s  \frac{d_i\,e_i }{\omega_i}=\frac{1}{3}
\end{equation}
is satisfied.
\end{thm}
\begin{proof}
If $\omega_i=\tomega_i$ for all $i=1,..,s$, it follows that
 $ \alpha_{ij}=\tilde \alpha_{ij}$ and $ \beta_{ij}=\tilde \beta_{ij}$ for
all $i,j=1,..,s$ and moreover $\gamma_i=\tilde \gamma_i$, $\tilde \delta_i=\delta_i$, $d_i=\tilde d_i$
and $e_i=\tilde e_i$ for all $i=1,..,s$.
Now, by \eref{3ordercond} we have
\begin{equation}
 \sum_{i=1}^s \omega_i\,\gamma_i^2=
\sum_{i=1}^s \,\frac{(\omega_i-d_i)^2}{\omega_i}=
\sum_{i=1}^s(\omega_i-2d_i+\frac{d_i^2}{\omega_i})=\frac{1}{3}
\end{equation}
and similarly,
\[
  \sum_{i=1}^s \omega_i\,\delta_i^2=  \frac{1}{3} \qquad \mbox{and}\qquad
\sum_{i=1}^s \omega_i\,\gamma_i\,\delta_i=  \frac{1}{3} \,.
\]
Furthermore,  the third order conditions for \eref{SRK-1} imply
\begin{eqnarray*}
 \sum_{i=1}^s \omega_i c_i\gamma_i
 =  \sum_{i=1}^s \omega_i c_i\left(\sum_{j=1}^s (\omega_j - \frac{\omega_j}{\omega_i}\, \tilde a_{ji})\right)
=\frac{1}{2} -  \sum_{i=1}^s\sum_{j=1}^s\omega_j\,\tilde a_{ji}\,c_i =\frac{1}{3}
\end{eqnarray*}
and similarly
\[
\sum_{i=1}^s \omega_i \tilde c_i\gamma_i =\frac{1}{3}, \qquad
\sum_{i=1}^s \omega_i \tilde c_i\delta_i =\frac{1}{3} \qquad \mbox{and}\qquad
\sum_{i=1}^s \omega_i  c_i \delta_i =\frac{1}{3}.
\]
Therefore, it also holds
\begin{eqnarray*}
 \sum_{i,j} \omega_i\beta_{ij} \gamma_j &=&
\sum_{i,j} \omega_i\,(\omega_j - \frac{\omega_j}{\omega_i}\, a_{ji})\,\gamma_j
= \frac{1}{2} - \sum_{j=1}^s \omega_j \gamma_j \sum_{i=1}^s a_{ji}
= \frac{1}{2} - \sum_{j=1}^s \omega_j \gamma_j c_j = \frac{1}{6}
\end{eqnarray*}
and
\[
 \sum_{i,j} \omega_i\beta_{ij} \delta_j =\frac{1}{6},\qquad
\sum_{i,j} \omega_i\alpha_{ij}  \gamma_j =\frac{1}{6}\qquad \mbox{and}\qquad
\sum_{i,j} \omega_i\alpha_{ij}\delta_j =\frac{1}{6}.
\]
For the corresponding coupling conditions we get
\begin{eqnarray*}
  \sum_{i,j} \omega_i\beta_{ij}c_j &=&
\sum_{i,j} \omega_i\,(\omega_j - \frac{\omega_j}{\omega_i}\, a_{ji})\,c_j
= \frac{1}{2} - \sum_{j=1}^s \omega_j c_j \sum_{i=1}^s a_{ji}
= \frac{1}{2} - \sum_{j=1}^s \omega_j c_j^2 =\frac{1}{6}
\end{eqnarray*}
and
\[
 \sum_{i,j} \omega_i\beta_{ij} \tilde c_j =\frac{1}{6},\qquad
\sum_{i,j} \omega_i\alpha_{ij}  c_j =\frac{1}{6}\qquad  \mbox{and}\qquad
\sum_{i,j} \omega_i\alpha_{ij} \tilde c_j =\frac{1}{6}.
\]
Finally the second set of associated coupling conditions follows by the fact that
\[
 \gamma_i=1-\sum_{j=1}^s \frac{\omega_j \tilde a_{ji}}{\omega_i}= 1-\frac{d_i}{\omega_i}
\qquad\mbox{and}\qquad
 \delta_i=1-\sum_{j=1}^s \frac{\omega_j a_{ji}}{\omega_i}= 1-\frac{e_i}{\omega_i}
\]
such that
\[
 \sum_{i,j} \omega_ia_{ij}\gamma_j= \sum_{i,j} \omega_ia_{ij}\left(1-\frac{d_j}{\omega_j}\right)
= \frac{1}{2}-\sum_{j=1}^s \frac{d_j}{\omega_j}\sum_{i=1}^s \omega_ia_{ij}=
 \frac{1}{2}-\sum_{j=1}^s \frac{d_j\,e_j}{\omega_j}=\frac{1}{6}.
\]
In analogous way we obtain
\[
 \sum_{i,j} \omega_i a_{ij} \delta_j =\frac{1}{6},\qquad
\sum_{i,j} \omega_i\tilde a_{ij}  \gamma_j =\frac{1}{6}\qquad \mbox{and}\qquad
\sum_{i,j} \omega_i\tilde a_{ij}\delta_j =\frac{1}{6}.
\]
\end{proof}

\subsection{Symplecticity}

Under appropriate conditions the equation \eref{osOCPc} can be explicitly solved and
thus eliminated from the optimality system \eref{osOCPc}:
  Assume that locally in the neighbourhood of a critical point $u \mapsto H_{uu}(y,u,p)$
 is invertible along the trajectory, then by the implicit function theorem, we  deduce the existence of
 a function $u=\varphi(y,p)$ that such that \eref{osOCPc} is satisfied
\cite{BonnansLaurent-Varin06,CHV08,Kaya10}.
Using the function $\varphi(y,p)$ the associated reduced Hamiltonian system is then
\begin{subequations}\label{hamiltonian}
\begin{eqnarray}
 \dot y&=& ~ ~ \,\mathcal H_p(y,p)= f(y,\varphi(y,p)) + g(y,\varphi(y,p))\label{hamiltonian-a}\\
 \dot p&=& - \mathcal H_y(y,p)= -f_y(y,\varphi(y,p))^T p -  g_y(y,\varphi(y,p))^T p,\label{hamiltonian-b}
\end{eqnarray}
\end{subequations}
where $\mathcal H(y,p):=H(y,\varphi(y,p),p)$. This system is a Hamiltionian
differential equation \cite{HNW93,CHV08}. It has been shown \cite{HNW93}, that
in general integration methods that preserve the geometric properties such as symplecticity are
more suitable to solve Hamiltionian systems \cite{CHV08}. However, for optimal control problems, the
advantage of symplectic integrators is not as clear, but there are  cases where they provide a
significant computational advantage \cite{CHV08}.


Consider the discrete Hamiltionian system
\begin{equation}\label{hamiltoniansys}
 \dot y_i=-H_{p_i}(y,p), \qquad \qquad \dot p_i=H_{y_i}(y,p) \qquad i=1,..,K
\end{equation}
It is known that the flow $\psi_t$ generated in the phase space $\real^K\times \real^K$ of
$(y,p)$ by the equations \eref{hamiltoniansys} is  \textsl{symplectic}, i.e.
it preserves the the differential 2-form
$
 \omega^2=\sum_{i=1}^K dy_i\wedge dp_i.
$
Here,  the differentials  $dy_i$ and $dp_i$ at the stage $i$ are computed as
derivatives of the flow with respect to the initial data. Preserving the differential 2--form is equivalent \cite{HNW93} to $\psi_t\ast \omega^2=\omega^2$ for the flow $\psi_t$. As an important consequence we have hat the flow is  volume and orientation preserving.  Numerical methods that preserve symplecticity are called symplectic. There exist a variety of papers on symplectic RK methods for example \cite{Jay96,SA91,SS88,CHV08}.

Theorem \ref{thm-symplectic} gives conditions such that  the discretization scheme \eref{SRK-1}
together with \eref{ARK-2} is symplectic for \eref{hamiltonian}. We assume that $u$ is given by
$u=\varphi(y,p)$ and \eref{Hu} is locally equivalent to $u_n^i=\tilde \varphi(Y^{(i)}_n,P^{(i)}_n,\tilde P^{(i)}_n)$.
Here,  $\tilde \varphi$ is the corresponding implicitly defined function that belongs to
\eref{aux-hamiltonian} and we identify $\tilde \varphi(y,p,\tilde p)=\varphi(y,p)$ for all $(p,\tilde p)$ with $p\equiv\tilde p$.

\begin{thm}\label{thm-symplectic}
Assuming that \eref{Hu} holds and provided that $\tilde \omega_i=\omega_i$ for all
$i=1,..,s$, then the discretization scheme \eref{SRK-1}
together with \eref{ARK-2} is a symplectic scheme for \eref{hamiltonian}.
\end{thm}
The proof of this theorem follows along the lines of the proof of Theorem 16.6 in \cite{HNW93}.
\begin{proof}
Consider the discretization scheme \eref{SRK-1}
together with \eref{ARK-2} as a discretization of \eref{hamiltonian}. Moreover, to simplify the notation
replace the time indices $n$ and $n+1$ by $0$ and $1$, respectively and let $K$ be the dimension
of $y_0$ or $p_0$, respectively. In order to prove the symplecticity of our discretization scheme,
we need to show that
\[
 \sum_{r=1}^K dy_1^r\wedge dp_1^r= \sum_{r=1}^K dy_0^r\wedge dp_0^r.
\]
First we  simplify each summand $dy_1^r\wedge dp_1^r-dy_0^r\wedge dp_0^r$ independently.
Using \eref{A1-symplectic} (see appendix), we  replace $dy_1^r$ and $dp_1^r$ and
obtain
\begin{subequations}\label{e1-symplectic}
\begin{eqnarray}
dy_1^r\wedge dp_1^r-dy_0^r\wedge dp_0^r&=&
-h\sum_{i=1}^s\tilde \omega_i (dy_0^r\wedge dQ_i^r) - h\sum_{i=1}^s \omega_i (dy_0^r\wedge dV_i^r)
\label{e1-symplectic-1}\\
 &&+h\sum_{i=1}^s\tilde \omega_i (dT_i^r\wedge dp_0^r) + h\sum_{i=1}^s \omega_i (dL_i^r\wedge dp_0^r)
\label{e1-symplectic-2}\\
&&-h^2\sum_{i,j}\tilde \omega_i\tilde \omega_j (dT_i^r\wedge dQ_j^r)
- h^2\sum_{i,j} \tilde \omega_i\omega_j (dT_i^r\wedge dV_j^r)
\label{e1-symplectic-3}\\
&&-h^2\sum_{i,j}\omega_i\tilde \omega_j (dL_i^r\wedge dQ_j^r)
- h^2\sum_{i,j}  \omega_i\omega_j (dL_i^r\wedge dV_j^r).
\label{e1-symplectic-4}
\end{eqnarray}
\end{subequations}
Next we replace $dy_0^r$ and $dp_0^r$ by terms resulting from \eref{A1-symplectic}
 and insert the corresponding terms into \eref{e1-symplectic}, then the rhs of
\eref{e1-symplectic} simplifies to
\begin{subequations}\label{e2-symplectic}
\begin{align}
dy_1^r\wedge dp_1^r-dy_0^r\wedge dp_0^r =   -h\sum_{i=1}^s\tilde \omega_i (dY_i^r\wedge dQ_i^r)
-h\sum_{i=1}^s\tilde \omega_i (dY_i^r\wedge dV_i^r)\\
 +h\sum_{i=1}^s\tilde \omega_i (dT_i^r\wedge d\tilde P_i^r)+
h\sum_{i=1}^s\tilde \omega_i (dL_i^r\wedge dP_i^r)
 -h^2\sum_{i,j}M^1_{ij} (dT_i^r\wedge dQ_j^r) \\
 -h^2\sum_{i,j}M^2_{ij} (dL_i^r\wedge dQ_j^r)
-h^2\sum_{i,j}M^3_{ij} (dT_i^r\wedge dV_j^r)
 -h^2\sum_{i,j}M^4_{ij} (dL_i^r\wedge dV_j^r)
\end{align}
\end{subequations}
Here, the entries of the matrices $M^1,M^2, M^3$ and $M^4$ are
\begin{align*}
 M^1_{ij}&= \tilde \omega_i\tilde\omega_j - \tilde \omega_j\tilde a_{ji} - \tilde \omega_i \tilde \alpha_{ij} \qquad \qquad
M^2_{ij}= \omega_i\tilde\omega_j - \tilde \omega_j a_{ji} - \omega_i \tilde \beta_{ij}\\[0.5em]
M^3_{ij}&= \tilde \omega_i\omega_j -  \omega_j\tilde a_{ji} - \tilde \omega_i  \alpha_{ij} \qquad \qquad
M^4_{ij}=  \omega_i\omega_j - \omega_j a_{ji} - \omega_i  \beta_{ij}.
\end{align*}
Hence, using  the coefficients $\tilde \alpha_{ij}$, $\alpha_{ij}$, $\tilde \beta_{ij}$ and
$\beta_{ij}$ all entries of $M^1,M^2, M^3$ and $M^4$ vanish. Using
 $\tilde  \omega_i=\omega_i$ for all $i=1,..,s$
we compute
\begin{align*}
\sum_{r=1}^K( dy_1^r\wedge dp_1^r- dy_0^r\wedge dp_0^r) =&
   h\left[-\sum_{i=1}^s\omega_i \sum_{r=1}^K  (dY_i^r\wedge (dQ_i^r +  dV_i^r))\right. \\
 &\qquad \left.+\sum_{i=1}^s\omega_i\sum_{r=1}^K (dT_i^r\wedge d\tilde P_i^r)+
(dL_i^r\wedge dP_i^r)\right].
\end{align*}
By \eref{A2-symplectic} we then obtain for the rhs of this equation
 \begin{align*}
-\; h\;\sum_{i=1}^s\omega_i \;\left[ \sum_{\ell,r=1}^K
\left(\sum_j\frac{\partial}{\partial y^r\partial y^\ell} f^j(Y^{(i)},u_i)  \tilde P_i^j +
\sum_j\frac{\partial}{\partial y^r\partial y^\ell}g^j(Y^{(i)},u_i)  P_i^j \right)
(dY^r_i\wedge dY^\ell_i)\right.\\[0.5em]
\left.\hspace{8em} + \left(\frac{\partial}{\partial y^r} f^\ell(Y^{(i)},u_i) (dY^r_i\wedge d\tilde P_i^\ell)
+ \frac{\partial}{\partial y^r} g^\ell(Y^{(i)},u_i) (dY^r_i\wedge dP_i^\ell)\right)\right]\\
+\;  h\;\sum_{i=1}^s\omega_i \;\left[ \sum_{\ell,r=1}^K
 \left(\frac{\partial}{\partial y^\ell} f^r(Y^{(i)},u_i) (dY^\ell_i\wedge d\tilde P_i^r)
+ \frac{\partial}{\partial y^\ell} g^r(Y^{(i)},u_i) \right)(dY^\ell_i\wedge dP_i^r)\right].
\end{align*}
Since
$ dY_i^r\wedge dY_i^\ell=-dY_i^\ell\wedge dY_i^r$ and $dY_i^r\wedge dY_i^r=0$
holds for the exterior product (cf. \cite{HNW93}),  the previous term vanishes and thus
\[
 \sum_{r=1}^K [(dy_1^r\wedge dp_1^r) -  (dy_0^r\wedge dp_0^r)]=0.
\]
\hfill\end{proof}

\section{Numerical example and implementation}
We consider the following problem taken from Hager \cite{Hager00}
\begin{align}
\min & \; \frac{1}2 \int_0^1 ( u^2 + 2 x^2 ) dt \quad \mbox{ subject to } \\
        &  \dot x(t) = \frac{1}2 x(t) + u(t), \; x(0) = 1.  \label{original-system}
\end{align}
The optimal solution is denoted by $(u^*,x^*)$ where
\begin{equation}
\label{opt-control-num}
u^*(t) = \frac{  2 ( \exp(3t)-\exp(3) ) }{ \exp( 3t / 2 ) ( 2 + \exp(3) ) }.
\end{equation}
To illustrate the numerical methods we reformulate the problem  as a singularly perturbed differential equation
\begin{subequations}\label{num-opt-pb}
\begin{align}
\min & \; c(1) \quad \mbox{subject to }\\
        & \dot c(t) = \frac{1}2  ( u^2(t) +  x^2(t) + 4 z^2(t)  ), \; c(0) = 0 \\
        & \dot x(t) = z(t) + u(t),  \; x(0)=1\\
        & \dot z(t) = \frac{1}\epsilon ( \frac{1}2 x(t) - z(t) ), \; z(0)=\frac{1}2
\end{align}
\end{subequations}
for some $\epsilon>0.$ Note that, as $\epsilon \to 0$, in (\ref{num-opt-pb}d)
we get $z(t)=x(t)/2$ and thus substituting into
(\ref{num-opt-pb}c) we recover system (\ref{original-system}). Taking
$$y=\begin{pmatrix}c \\ x \\ z \end{pmatrix}, \;
f(y,u)= \begin{pmatrix}\displaystyle\frac{1}2(u^2+x^2+4z^2) \\  z+ u  \\  0 \end{pmatrix}, \;
g(y) = \frac{1}{\epsilon }\begin{pmatrix}Ê0 \\ 0\\ \displaystyle\frac{1}2x -z  \end{pmatrix}, $$
we obtain  a system of
the form (\ref{OCP}) for a scalar valued control $u(t).$  We discretize
the system IMEX methods. We report on results for the second--order L--stable IMEX SSP2
scheme (\ref{tab02}), the second--order globally stiffly accurate IMEX--GSA (\ref{table-gsa})
and the third--order IMEX--SA3 (\ref{3rd-order}) scheme. In all cases we consider
 an equidistant grid on $[0,1]$ with $n=1,\dots,N$ gridpoints. In the following
we set $u=(u_n)_{n=1}^N$ where $u_n=u(t_n)$ and similarly for $x$ and $c.$
  In order to numerically solve the optimality conditions for (\ref{num-opt-pb}) it remains
to solve  equations (\ref{SRK-1}),(\ref{ARK-2}) and (\ref{Hu}). We do not
present the detailed formulas but have some remarks concerning the implementation.

 In view of
the stiff source in $g$ it is advantegous to solve instead of (\ref{SRK-1})
the equivalent system (\ref{SRK-2}). As mentioned below Lemma \ref{Lem2} the
adjoint equation (\ref{ARK-2}) has to be solved backwards in time. It is therefore
advantegous to use the equivalent formulation (\ref{ARK-1}). Note that the discretized terminal
condition for the adjoint scheme is $p_N=(1,0,0).$
Furthermore, we assume for simplicity that
the control $u$ is the same on each stage. Then, equation (\ref{Hu})
reads
\begin{equation}\label{Hu-num}
\sum\limits_{k=1}^s \tilde \omega_k f_u(Y^{(k)}_n, u_n) \tilde P^{(k)}_n +
\omega_k g_u(Y^{(k)}_n, u_n)  P^{(k)}_n = 0.
\end{equation}
and
$$ \tilde P^{(i)}_n  h \tilde \omega_i = \tilde \xi_n^{(i)}, \;    P^{(i)}_n  h  \omega_i =  \xi_n^{(i)},$$
where $\tilde \xi$ and $\xi$ are the solution to (\ref{ARK-1}). The optimality system
 (\ref{SRK-1}),(\ref{ARK-2}), (\ref{Hu})
 is solved by a block Gauss--Seidel method:
Provided we know
$u$ on all grid points we solve (\ref{SRK-2}) to obtain the state $y$.
Having $u$ and $y$ at hand we can solve (\ref{ARK-1}) in order to
obtain $p$ and subsequently $\tilde \xi$ and $\xi.$
However, for an arbitrary $u$ equation (\ref{Hu-num}) will not hold true.
We therefore use a nonlinear root finding method $F(u)=0$.  We set
\begin{equation}\label{F}
F(u) = \left\| ( F_n( u ) )_{n=1}^N  \right\|^2_2, \; F_n(u) =  \sum\limits_{k=1}^s
 f_u(Y^{(k)}_n, u_n) \tilde \xi^{(k)}_n +
 g_u(Y^{(k)}_n, u_n)  \xi^{(k)}_n,
 \end{equation}
 where $Y^{(k)}_n, \tilde \xi^{(k)}_n$ and $\xi^{(k)}_n$ are all dependent on $u$ through
(\ref{SRK-2}) and (\ref{ARK-1}), respectively.
Since in the
example $g_y$ is independent of $y$ we can solve the adjoint equations (\ref{ARK-1})
more efficiently. We  use the fact that
$$Y^i_n = y(t_n + \tilde c_i \Delta t) + O( (\Delta t)^p ),$$
where $p$ is the order of the explicit part of the IMEX scheme . Hence, instead of computing
at every time step in the adjoint equation the values $Y^i_n$ for $i=1,\dots,s$
we interpolate the given data $(y_n)_n$ using a second--and third--order accurate
interpolation, respectively.
\par
We study the dependence of $F(u^*)$ on grid size and value of $\epsilon$ for the
control $u^*$ given by (\ref{opt-control-num}). This control is optimal only in the case
of $\epsilon=0$.    In order to observe the convergence
rates we compute a fine grid solution on $N=640$ grid points. The corresponding
states and adjoints are denoted by $(c^*,x^*,z^*)$ and $(p_1^*,p_2^*,p_3^*),$ respectively.
Note that due to the particular structure of the problem we always have $p_1^*(t)=1$ and therefore we
do not report this quantity in the tables below.  We denote by $F^* =(F_n(u^*))$ and we not necessarily
have  $\| F^* \|_\infty  =0$  since  the control $u^*$  is not optimal for the problem (\ref{num-opt-pb}) in case
 $\epsilon>0$. As in \cite{Hager00} we also report the $L^\infty-$error in the state $x(t)$,
and relaxation variable $z(t)$ obtained on a finer grid. Convergence results
for the IMEX--SA3 and IMEX--GSA are given in Table \ref{res-SA3} and Table \ref{res-GSA}, respectively.
  We observe that the convergence properties in the state remain independent on $\epsilon.$
 Since the IMEX--SA3  scheme is not stiffly accurate we loose third--order convergence in the relaxation
variable $z$ when  $\epsilon$ is underresolved. Contrary,
we observe second--order convergence also for small $\epsilon$ for
 the stiffly accurate IMEX--GSA scheme.

Further, we study the convergence behavior of the nonlinear root finding
method. Initially, we set $u(t)=1$ and subsequently solve $F(u)=0$
using a standard black--box root finding method of Matlab with termination
tolerance $1.e-08$ on $F.$ Even so $F( (u_n)_n )$ is zero up to machine precision in the optimization
procedure there remains a difference in the computed trajectory $(x_n)_n$
compared with $x^*.$  The results for IMEX--SSP2
are given in Table \ref{res-SSP2}. We observe that the $L^2-$difference between
analytical and numerically computed trajectory and control decreases with  grid
size.

\section{Summary and conclusions}

We investigated the application of IMEX Runge-Kutta methods to
optimal control problems. In particulare we focused on order
conditions and conditions for symplecticity. We studied the
adjoint equations and established a commutative diagram for
optimization and discretization. In particular cases previous
results for single Runge-Kutta schemes could be recovered
\cite{BonnansLaurent-Varin06, Hager00}. Examples of up to third
order IMEX Runge-Kutta methods are given and numerical results for
a sample problem have been presented.


\section*{Appendix}
\subsection*{A1. Addenda on the proof of Lemma \ref{Lem1}}\label{appendix-1}
The Lagrangian function of \eref{DOP1} is given by
\begin{eqnarray*}
 \mathcal L(y, Y,p,\zeta)&=& j(y_N)+p^0\cdot(y_0-y^0)\nonumber\\[0.3em]
&&\quad +\sum_{n=0}^{N-1}
\left[ p_{n+1} ^T\left(- y_{n+1}+ y_n +
h\sum_{i=1}^s \tomega_if(Y_n^{(i)},u_n^i) + h\sum_{i=1}^s \omega_ig(Y_n^{(i)},u_n^i)\right)\right.\\[0.5em]
&&\left.\qquad  + \sum_{i=1}^s (\zeta_n^{(i)})^T(-Y^{(i)}_n + y_n + h\sum_{j=1}^{s} \ta_{ij}f(Y_n^{(j)},u_n^j) + h\sum_{j=1}^s
a_{ij}g(Y_n^{(j)},u_n^j))\right]\,.\nonumber
\end{eqnarray*}

To simplify the notation in the proof of Theorem \ref{thm-symplectic}, the
time indices $n$ and $n+1$ are replaced by $0$ and $1$, respectively, for the discrete approximations
$y_n$ and $p_n$ and  they are ommitted for the intermediate states,
which are denoted by vectors $Y_i\in \real^K$, $P_i\in \real^K$ or $\tilde P_i\in \real^K$,
respectively, with $i$ being the index for the stages $1,..,s$.
In this notation the schemes \eref{SRK-1} and \eref{ARK-2} read for $ i=1,..,s$
 \begin{align*}
Y_i &=y_0 + h\sum_{j=1}^{s} \ta_{ij}f(Y_j,u_j) + h\sum_{j=1}^s
a_{ij}g(Y_j,u_j), \\
y_{1}&= y_0 + h\sum_{i=1}^s \tomega_if(Y_i,u_i) + h\sum_{i=1}^s \omega_ig(Y_i,u_i)
 \end{align*}
and for $i=1,\dots,s$
 \begin{eqnarray*}
\tilde P_i&  =&p_0 - h\sum_{j=1}^s \tilde \alpha_{ij}\,f_y(Y_j,u_j)^T \tilde P_j
- h\sum_{j=1}^s \alpha_{ij}\,g_y(Y_j,u_j)^T P_j \\
P_i&  =&p_0-h\sum_{j=1}^s \tilde \beta_{ij}\,f_y(Y_j,u_j)^T\tilde P_j
- h\sum_{j=1}^s \beta_{ij}\,g_y(Y_j,u_j)^T P_j  \\
p_{1}&=& p_{0} - h\sum_{i=1}^s \tomega_i \,f_y(Y_i,u_i)^T \tilde P_i
-\sum_{i=1}^s \omega_i \,g_y(Y_i,u_i)^T P_i.
 \end{eqnarray*}
The one-forms $dy_1^r:\real^{2K}\rightarrow \real$ and $dY_i^r:\real^{2K}\rightarrow \real$
 used in the proof of Theorem \ref{thm-symplectic}
are defined by
\[
 z\mapsto\frac{\partial y_1^r}{\partial (y_0,p_0)}z \qquad \mbox{and}
\qquad z\mapsto\frac{\partial Y_i^r}{\partial (y_0,p_0)}z ,
\]
respectively and similarly also $dp_1^r$, $dP_i^r$ and $d\tilde P_i^r$,
where $y_1^r$,  $p_1^r$, $Y^r_i$, $P^r_i$ and $\tilde P^r_i$,  denote the
$r$th component of the corresponding vector.
By differentiation of the above equations with respect to
$(y_0,p_0)$ and using the linearity of the differential we obtain
\begin{subequations}\label{A1-symplectic}
\begin{align}
 dY_i^r = dy_0^r + h \sum_{j=1}^s \tilde a_{ij} dT_j^r +h \sum_{j=1}^s  a_{ij} dL_j^r ,
 dy_1^r = dy_0^r + h \sum_{i=1}^s \tilde \omega_i dT_i^r +h \sum_{i=1}^s  \omega_i dL_i^r
\\
 d\tilde P_i^r = dp_0^r - h \sum_{j=1}^s \tilde \alpha_{ij} dQ_j^r -h \sum_{j=1}^s  \alpha_{ij} dV_j^r,
dP_i^r = dp_0^r - h \sum_{j=1}^s \tilde \beta_{ij} dQ_j^r - h \sum_{j=1}^s  \beta_{ij} dV_j^r
\\
dp_1^r = dp_0^r + h \sum_{i=1}^s \tilde \omega_i dQ_i^r +h \sum_{i=1}^s  \omega_i dV_i^r
\end{align}
\end{subequations}
for $i=1,..,s$ and $r=1,..,K$, where
\begin{subequations}\label{A2-symplectic}
\begin{align}
 dT_i^r = \sum_{\ell=1}^K \frac{\partial}{\partial y^\ell} f^r(Y_i,u_i) dY_i^\ell, \;
  dL_i^r =  \sum_{\ell=1}^K \frac{\partial}{\partial y^\ell} g^r(Y_i,u_i) dY_i^\ell, \;
\\
 dQ_i^r = \sum_{j=1}^K \sum_{\ell=1}^K\frac{\partial^2}{\partial y^\ell\partial y^r} f^j(Y_i,u_i)
\tilde P_i^j, \; dY_i^\ell+ \sum_{j=1}^K\frac{\partial}{\partial y^r} f^j(Y_i,u_i)
 d\tilde P_i^j, \; \\
dV_i^r = \sum_{j=1}^K \sum_{\ell=1}^K\frac{\partial^2}{\partial y^\ell\partial y^r} g^j(Y_i,u_i)
 P_i^j, \; dY_i^\ell+ \sum_{j=1}^K\frac{\partial}{\partial y^r} g^j(Y_i,u_i)
dP_i^j.
\end{align}
\end{subequations}

\subsection*{A2. Examples of IMEX schemes}\label{appendix-2}
We use the following convention for the names of the schemes: Name$(k,\sigma_E,\sigma_I)$ where $k$ is the order, $\sigma_E$ the number of levels in the explicit scheme and $\sigma_I$ the number of levels in the implicit scheme.

A two stage second order IMEX method, where the implicit part is L-stable, is
given by Pareschi and Russo in \cite{PR03}.
\begin{table}[!h]
  \begin{center}
 \begin{tabular}{ccc}
    $(\tilde A,\tilde \omega)$:\hspace*{0.2cm}\begin{tabular}{c|cc}
      $0$ & $0$ & $0$\\
      $1$ & $1$ & $0$\\
      \hline
      & $1/2$ & $1/2$
    \end{tabular}
    & \hspace*{0.025\textwidth} &
    $(A,\omega)$:\hspace*{0.2cm}\begin{tabular}{c|cc}
      $\gamma$ & $\gamma$ & $0$\\
      $1-\gamma$ & $1-2\gamma$ & $\gamma$\\
      \hline
      & $1/2$ & $1/2$
    \end{tabular}
  \end{tabular}
\end{center}
  \caption{ IMEX--SSP$2(2,2,2)$ is a second--order IMEX scheme.
The factor $\gamma$ is given by $\gamma=1-1/{\sqrt{2}}$.}
\label{tab02}
\end{table}
Since $\tomega=\omega$ the second order conditions for the
additive RK scheme are directly satisfied.  To avoid loss of accuracy in stiff problems, in order to compute a globally stiffly accurate method ($\ta_{sj}=\tomega_j$ and $a_{sj}=\omega_j$, $j=1,\ldots,s$ see \cite{BPR11}), we are forced to take $\tomega\neq\omega$, $\tomega_{\nu}=0$ and impose the additional second order conditions (\ref{2ordercond}). In this case at least $4$ levels are required. An example is IMEX-GSA$(2,3,4)$ reported below

\begin{table}[!h]
  \begin{center}
 \begin{tabular}{ccc}
    $(\tilde A,\tilde \omega)$:\hspace*{0.2cm}\begin{tabular}{c|cccc}
      $0$ & $0$ & $0$&$0$&$0$\\
     $3/2$ & $3/2$ & $0$&$0$&$0$\\
    $1/2$ & $5/6$ & $-1/3$&$0$&$0$\\
    $1$ & $1/3$&$1/6$ & $1/2$&$0$\\
      \hline
      & $1/3$&$1/6$ & $1/2$&$0$
    \end{tabular}
    & \hspace*{0.025\textwidth} &
    $(A,\omega)$:\hspace*{0.2cm}\begin{tabular}{c|cccc}
     $1/2$ & $1/2$ & $0$&$0$&$0$\\
      $5/4$ & ${3}/{4}$ & ${1}/{2}$&$0$&$0$\\
       $1/4$ & $-1/4$ & $0$&$1/2$&$0$\\
       $1$&$1/6$&$-1/6$&$1/2$&$1/2$\\
      \hline
       & $1/6$&$-1/6$&$1/2$&$1/2$
    \end{tabular}
  \end{tabular}
\end{center}
  \caption{IMEX--GSA$(2,3,4)$ is a second order globally stiffly accurate scheme.}
  \label{table-gsa}
\end{table}

Moreover, it can be shown
that for $\tomega=\omega$ no three stages IMEX--RK that satisfies
the additional third-order conditions in Theorem \ref{thm-3order} with
A-stable implicit integrator exist. Here, we report
three stage third order IMEX--RK method such that the explicit scheme
corresponds to the
third-order scheme given in \cite{Hager00}.
\begin{table}[!h]
  \begin{center}
 \begin{tabular}{ccc}
    $(\tilde A,\tilde \omega)$:\hspace*{0.2cm}\begin{tabular}{c|ccc}
      $0$ & $0$ & $0$&$0$\\
     $1/2$ & $1/2$ & $0$&$0$\\
    $1$ & $-1$ & $2$&$0$\\
      \hline
      & $1/6$ & $2/3$&$1/6$
    \end{tabular}
    & \hspace*{0.025\textwidth} &
    $(A,\omega)$:\hspace*{0.2cm}\begin{tabular}{c|ccc}
     $0$ & $0$ & $0$&$0$\\
      $1/2$ & ${1}/{4}$ & ${1}/{4}$&$0$\\
       $1$ & $0$ & $1$&$0$\\
      \hline
       & $1/6$ & $2/3$&$1/6$
    \end{tabular}
  \end{tabular}
\end{center}
  \caption{ IMEX--HAG$(3,3,3)$ is a third--order IMEX scheme, where
$(\tilde A,\tilde \omega)$ corresponds to the third-order scheme given by Hager in \cite{Hager00}.}
\end{table}

Finally we present a third order scheme which uses $4$ levels in order to achieve better stability properties in the implicit integrator.
\begin{table}[!h]
  \begin{center}
 \begin{tabular}{ccc}
    $(\tilde A,\tilde \omega)$:\hspace*{0.2cm}\begin{tabular}{c|cccc}
      $0$ & $0$ & $0$&$0$&$0$\\
     $2/3$ & $2/3$ & $0$&$0$&$0$\\
    $1$ & $3/4$ & $1/4$&$0$&$0$\\
    $1$ & $1/4$&$3/4$ & $0$&$0$\\
      \hline
      & $1/4$&$3/4$ & $-1/2$&$1/2$
    \end{tabular}
    & \hspace*{0.025\textwidth} &
    $(A,\omega)$:\hspace*{0.2cm}\begin{tabular}{c|cccc}
     $0$ & $0$ & $0$&$0$&$0$\\
      $2/3$ & $-{1}/{3}$ & ${1}$&$0$&$0$\\
       $1$ & $-1/4$ & $1/4$&$1$&$0$\\
       $1$&$1/4$&$3/4$&$-1/2$&$1/2$\\
      \hline
       & $1/4$&$3/4$&$-1/2$&$1/2$
    \end{tabular}
  \end{tabular}
\end{center}
  \caption{IMEX--SA$(3,4,4)$ is a four stages, third order IMEX scheme.}
  \label{3rd-order}
\end{table}

Note that all IMEX schemes of type ARS have $\omega_1=0$,
such that they do not satisfy the condition $\tomega\neq0$ and $\omega\neq0$.
However if they are of the particular structure of \cite{ARS97} with $\omega_j\neq 0$ for $j\neq1$
and $\tomega\neq 0$, then we can still find a variable transformation such that the
conclusion of Proposition \ref{ARK-reformulation} holds (see Remark in Section \ref{sec2}).


\begin{sidewaystable}[h]
 \begin{center}
 \begin{tabular}{|c|c|c|c|c|c|c|}
  \hline
     $\epsilon$ & $N$ & $\|F_n(u^*) - F_n(u_n)\|_\infty$
     &  $ \| x^*- (x_{n})_n \|_\infty$  (Ratio) &  $ \| z^*- (z_{n})_n \|_\infty$      (Ratio)
     &  $ \| p_2^*- (p_{2,n})_n \|_\infty$  (Ratio) &  $ \| p_3^*- (p_{3,n})_n \|_\infty$      (Ratio)   \\
\hline
 1e+01 & 10 &  9.1756e-02 (0.00) & 3.1890e-05  (0.00) & 2.0043e-06 (0.00) & 1.2265e-05 (0.00) & 3.4885e-05 (0.00) \\
 1e+01 & 20 &  4.5186e-02 (2.03) & 3.2932e-06  (9.68) & 2.0679e-07 (9.69) & 1.3011e-06 (9.43) & 3.6568e-06 (9.54) \\
 1e+01 & 40 &  2.1851e-02 (2.07) & 3.7927e-07  (8.68) & 2.3698e-08 (8.73) & 1.4989e-07 (8.68) & 4.1983e-07 (8.71) \\
 1e+01 & 80 &  1.0192e-02 (2.14) & 4.5576e-08  (8.32) & 2.8387e-09 (8.35) & 1.8000e-08 (8.33) & 5.0267e-08 (8.35) \\
 1e+01 & 160 &  4.3665e-03 (2.33) & 5.5187e-09  (8.26) & 3.4314e-10 (8.27) & 2.1781e-09 (8.26) & 6.0731e-09 (8.28) \\
 1e+01 & 320 &  1.4553e-03 (3.00) & 6.0801e-10  (9.08) & 3.7774e-11 (9.08) & 2.3984e-10 (9.08) & 6.6831e-10 (9.09) \\

 \hline
 1e+00 & 10 &  6.0739e-02 (0.00) & 7.1856e-05  (0.00) & 1.1252e-05 (0.00) & 1.8045e-04 (0.00) & 3.6940e-04 (0.00) \\
 1e+00 & 20 &  2.9807e-02 (2.04) & 7.5380e-06  (9.53) & 1.3389e-06 (8.40) & 2.0218e-05 (8.93) & 4.2999e-05 (8.59) \\
 1e+00 & 40 &  1.4394e-02 (2.07) & 8.6958e-07  (8.67) & 1.6324e-07 (8.20) & 2.3966e-06 (8.44) & 5.1837e-06 (8.30) \\
 1e+00 & 80 &  6.7095e-03 (2.15) & 1.0444e-07  (8.33) & 2.0088e-08 (8.13) & 2.9141e-07 (8.22) & 6.3527e-07 (8.16) \\
 1e+00 & 160 &  2.8737e-03 (2.33) & 1.2639e-08  (8.26) & 2.4586e-09 (8.17) & 3.5473e-08 (8.22) & 7.7621e-08 (8.18) \\
 1e+00 & 320 &  9.5761e-04 (3.00) & 1.3919e-09  (9.08) & 2.7217e-10 (9.03) & 3.9175e-09 (9.05) & 8.5874e-09 (9.04) \\

 \hline
 1e-01 & 10 &  1.2860e-02 (0.00) & 1.8886e-04  (0.00) & 1.5712e-03 (0.00) & 2.4410e-03 (0.00) & 5.2256e-03 (0.00) \\
 1e-01 & 20 &  6.2772e-03 (2.05) & 2.6021e-05  (7.26) & 2.2140e-04 (7.10) & 3.8175e-04 (6.39) & 8.2394e-04 (6.34) \\
 1e-01 & 40 &  3.0272e-03 (2.07) & 3.6994e-06  (7.03) & 3.2117e-05 (6.89) & 5.6588e-05 (6.75) & 1.1973e-04 (6.88) \\
 1e-01 & 80 &  1.4106e-03 (2.15) & 5.0436e-07  (7.33) & 4.4293e-06 (7.25) & 7.8629e-06 (7.20) & 1.6467e-05 (7.27) \\
 1e-01 & 160 &  6.0415e-04 (2.33) & 6.5544e-08  (7.69) & 5.7612e-07 (7.69) & 1.0298e-06 (7.64) & 2.1441e-06 (7.68) \\
 1e-01 & 320 &  2.0132e-04 (3.00) & 7.4775e-09  (8.77) & 6.5816e-08 (8.75) & 1.1802e-07 (8.73) & 2.4499e-07 (8.75) \\

 \hline
 1e-04 & 10 &  7.1929e-05 (0.00) & 8.6575e-05  (0.00) & 8.7328e-05 (0.00) & 6.3322e-03 (0.00) & 1.2066e-02 (0.00) \\
 1e-04 & 20 &  1.0419e-05 (6.90) & 9.0398e-06  (9.58) & 4.3047e-05 (2.03) & 1.3610e-03 (4.65) & 2.6616e-03 (4.53) \\
 1e-04 & 40 &  3.7855e-06 (2.75) & 1.0419e-06  (8.68) & 4.8027e-05 (0.90) & 3.1411e-04 (4.33) & 6.2276e-04 (4.27) \\
 1e-04 & 80 &  1.6983e-06 (2.23) & 1.2524e-07  (8.32) & 5.0930e-05 (0.94) & 7.7166e-05 (4.07) & 1.5282e-04 (4.08) \\
 1e-04 & 160 &  7.2403e-07 (2.35) & 1.5191e-08  (8.24) & 5.8513e-05 (0.87) & 6.1901e-05 (1.25) & 1.2381e-04 (1.23) \\
 1e-04 & 320 &  2.4114e-07 (3.00) & 7.1227e-09  (2.13) & 7.0310e-05 (0.83) & 7.0572e-05 (0.88) & 1.4115e-04 (0.88) \\

 \hline
\end{tabular}
\end{center}
  \caption{ $u^*$ is given by equation (\ref{opt-control-num}) and computed on a fine grid of $N=640.$
   The corrresponding
  state are $(c^*,x^*,z^*)$. The state $(c_n,x_n,z_n)_n$ is the discrete solution to (\ref{SRK-1}) using  $u_n$ given by equation (\ref{opt-control-num}) on the corresponding grid $N$. The values of $F$ for $u^*$ and $u_n$ are given by  (\ref{F}) on the respective grids where the adjoint states are obtained through (\ref{ARK-2})
  using IMEX--SA3. The value in the brackets is the residual of the current
  divided by the residual of the previous result.  }
\label{res-SA3}
\end{sidewaystable}

\begin{sidewaystable}[h]
 \begin{center}
 \begin{tabular}{|c|c|c|c|c|c|c|}
  \hline
     $\epsilon$ & $N$ & $\|F_n(u^*) - F_n(u_n)\|_\infty$
     &  $ \| x^*- (x_{n})_n \|_\infty$  (Ratio) &  $ \| z^*- (z_{n})_n \|_\infty$      (Ratio)
     &  $ \| p_2^*- (p_{2,n})_n \|_\infty$  (Ratio) &  $ \| p_3^*- (p_{3,n})_n \|_\infty$      (Ratio)   \\
\hline
 1e+01 & 10 &  9.2301e-02 (0.00) & 1.7243e-03  (0.00) & 2.1503e-04 (0.00) & 5.1541e-04 (0.00) & 6.5622e-04 (0.00) \\
 1e+01 & 20 &  4.5250e-02 (2.04) & 3.9831e-04  (4.33) & 4.9394e-05 (4.35) & 1.1331e-04 (4.55) & 1.6797e-04 (3.91) \\
 1e+01 & 40 &  2.1859e-02 (2.07) & 9.5630e-05  (4.17) & 1.1824e-05 (4.18) & 2.6590e-05 (4.26) & 4.2111e-05 (3.99) \\
 1e+01 & 80 &  1.0193e-02 (2.14) & 2.3345e-05  (4.10) & 2.8640e-06 (4.13) & 6.3807e-06 (4.17) & 1.0420e-05 (4.04) \\
 1e+01 & 160 &  4.3667e-03 (2.33) & 5.6848e-06  (4.11) & 6.7536e-07 (4.24) & 1.4979e-06 (4.26) & 2.4821e-06 (4.20) \\
 1e+01 & 320 &  1.4553e-03 (3.00) & 1.3305e-06  (4.27) & 1.3445e-07 (5.02) & 2.9758e-07 (5.03) & 4.9649e-07 (5.00) \\

 \hline
 1e+00 & 10 &  6.1450e-02 (0.00) & 2.5447e-03  (0.00) & 1.3985e-03 (0.00) & 1.6246e-03 (0.00) & 4.7224e-04 (0.00) \\
 1e+00 & 20 &  2.9890e-02 (2.06) & 5.7481e-04  (4.43) & 3.2585e-04 (4.29) & 3.8663e-04 (4.20) & 1.1769e-04 (4.01) \\
 1e+00 & 40 &  1.4404e-02 (2.08) & 1.3654e-04  (4.21) & 7.8532e-05 (4.15) & 9.4343e-05 (4.10) & 2.9295e-05 (4.02) \\
 1e+00 & 80 &  6.7107e-03 (2.15) & 3.2950e-05  (4.14) & 1.9085e-05 (4.11) & 2.3378e-05 (4.04) & 7.2297e-06 (4.05) \\
 1e+00 & 160 &  2.8739e-03 (2.34) & 7.7608e-06  (4.25) & 4.5075e-06 (4.23) & 5.8078e-06 (4.03) & 1.7202e-06 (4.20) \\
 1e+00 & 320 &  9.5763e-04 (3.00) & 1.6021e-06  (4.84) & 8.9805e-07 (5.02) & 1.3160e-06 (4.41) & 3.4393e-07 (5.00) \\

 \hline
 1e-01 & 10 &  1.4564e-02 (0.00) & 3.5004e-03  (0.00) & 2.8345e-03 (0.00) & 1.4019e-02 (0.00) & 2.0384e-02 (0.00) \\
 1e-01 & 20 &  6.5111e-03 (2.24) & 5.4949e-04  (6.37) & 6.2277e-04 (4.55) & 4.1263e-03 (3.40) & 6.7796e-03 (3.01) \\
 1e-01 & 40 &  3.0602e-03 (2.13) & 8.5496e-05  (6.43) & 1.4232e-04 (4.38) & 1.1543e-03 (3.57) & 1.9908e-03 (3.41) \\
 1e-01 & 80 &  1.4151e-03 (2.16) & 1.3647e-05  (6.26) & 3.6618e-05 (3.89) & 3.0554e-04 (3.78) & 5.3793e-04 (3.70) \\
 1e-01 & 160 &  6.0473e-04 (2.34) & 2.7495e-06  (4.96) & 9.2319e-06 (3.97) & 7.5515e-05 (4.05) & 1.3418e-04 (4.01) \\
 1e-01 & 320 &  2.0138e-04 (3.00) & 7.8573e-07  (3.50) & 1.8843e-06 (4.90) & 1.5388e-05 (4.91) & 2.7456e-05 (4.89) \\

 \hline
 1e-04 & 10 &  2.6263e-03 (0.00) & 7.1375e-03  (0.00) & 3.5731e-03 (0.00) & 9.8943e-03 (0.00) & 9.5557e-05 (0.00) \\
 1e-04 & 20 &  3.0503e-04 (8.61) & 1.6575e-03  (4.31) & 8.3367e-04 (4.29) & 2.3649e-03 (4.18) & 1.1131e-04 (0.86) \\
 1e-04 & 40 &  3.8390e-05 (7.95) & 3.9915e-04  (4.15) & 2.0443e-04 (4.08) & 6.0419e-04 (3.91) & 1.1350e-04 (0.98) \\
 1e-04 & 80 &  5.5770e-06 (6.88) & 9.7706e-05  (4.09) & 5.3299e-05 (3.84) & 1.7626e-04 (3.43) & 1.0612e-04 (1.07) \\
 1e-04 & 160 &  1.0622e-06 (5.25) & 2.3832e-05  (4.10) & 1.7548e-05 (3.04) & 6.5098e-05 (2.71) & 8.7494e-05 (1.21) \\
 1e-04 & 320 &  2.4145e-07 (4.40) & 5.1861e-06  (4.60) & 1.4779e-05 (1.19) & 2.6513e-05 (2.46) & 5.3256e-05 (1.64) \\

 \hline
 1e-08 & 10 &  2.6200e-03 (0.00) & 7.1442e-03  (0.00) & 3.5721e-03 (0.00) & 9.8738e-03 (0.00) & 3.3238e-09 (0.00) \\
 1e-08 & 20 &  3.0072e-04 (8.71) & 1.6597e-03  (4.30) & 8.2987e-04 (4.30) & 2.3297e-03 (4.24) & 1.4556e-09 (2.28) \\
 1e-08 & 40 &  3.6025e-05 (8.35) & 3.9927e-04  (4.16) & 1.9964e-04 (4.16) & 5.6437e-04 (4.13) & 6.6291e-10 (2.20) \\
 1e-08 & 80 &  4.4018e-06 (8.18) & 9.6931e-05  (4.12) & 4.8466e-05 (4.12) & 1.3746e-04 (4.11) & 2.9937e-10 (2.21) \\
 1e-08 & 160 &  5.3710e-07 (8.20) & 2.2882e-05  (4.24) & 1.1441e-05 (4.24) & 3.2501e-05 (4.23) & 1.2566e-10 (2.38) \\
 1e-08 & 320 &  5.9391e-08 (9.04) & 4.5577e-06  (5.02) & 2.2788e-06 (5.02) & 6.4837e-06 (5.01) & 4.1090e-11 (3.06) \\
\hline
\end{tabular}
\end{center}
  \caption{ $u^*$ is given by equation (\ref{opt-control-num}) and computed on a fine grid of $N=640.$
   The corrresponding
  state are $(c^*,x^*,z^*)$. The state $(c_n,x_n,z_n)_n$ is the discrete solution to (\ref{SRK-1}) using  $u_n$ given by equation (\ref{opt-control-num}) on the corresponding grid $N$. The values of $F$ for $u^*$ and $u_n$ are given by  (\ref{F}) on the respective grids where the adjoint states are obtained through (\ref{ARK-2})
  using IMEX--GSA. The value in the brackets is the residual of the current
  divided by the residual of the previous result.  }
\label{res-GSA}
\end{sidewaystable}

\begin{table}[h]
  \begin{center}
 \begin{tabular}{|c|c|c|c|c|}
  \hline
    $N$ & $\epsilon$ & $\|F_n(u_n)_n\|^2_2$ &  $ \| x^*- (x_n)_n \|_\infty$  &  $ \| u^*- (u_n)_n \|_\infty$   \\
  \hline
 10 & 0e+00 & 9.3435e-16 & 1.1661e-01 (0.00) & 8.1212e-02 (0.00) \\
 20 & 0e+00 & 5.7291e-15 & 5.4444e-02 (2.14) & 3.5548e-02 (2.28) \\
 40 & 0e+00 & 2.1384e-14 & 2.5674e-02 (2.12) & 1.6495e-02 (2.16) \\
 80 & 0e+00 & 4.8134e-13 & 1.1822e-02 (2.17) & 7.9952e-03 (2.06) \\
 160 & 0e+00 & 5.2921e-13 & 5.0243e-03 (2.35) & 3.7065e-03 (2.16) \\
 320 & 0e+00 & 4.9687e-11 & 1.6345e-03 (3.07) & 1.4217e-03 (2.61) \\
 \hline
 10 & 1e-02 & 9.2474e-16 & 1.1649e-01 (0.00) & 8.0766e-02 (0.00) \\
 20 & 1e-02 & 4.6022e-15 & 5.4428e-02 (2.14) & 3.5485e-02 (2.28) \\
 40 & 1e-02 & 1.6138e-13 & 2.5673e-02 (2.12) & 1.6428e-02 (2.16) \\
 80 & 1e-02 & 4.0505e-13 & 1.1818e-02 (2.17) & 7.9385e-03 (2.07) \\
 160 & 1e-02 & 4.9368e-13 & 5.0148e-03 (2.36) & 3.6611e-03 (2.17) \\
 320 & 1e-02 & 2.1669e-11 & 1.6241e-03 (3.09) & 1.4336e-03 (2.55) \\
 \hline
 10 & 1e-01 & 8.7718e-15 & 1.1655e-01 (0.00) & 7.9657e-02 (0.00) \\
 20 & 1e-01 & 6.4816e-15 & 5.4631e-02 (2.13) & 3.5508e-02 (2.24) \\
 40 & 1e-01 & 2.9944e-14 & 2.5813e-02 (2.12) & 1.6218e-02 (2.19) \\
 80 & 1e-01 & 5.3096e-14 & 1.1900e-02 (2.17) & 7.6948e-03 (2.11) \\
 160 & 1e-01 & 1.0236e-13 & 5.0614e-03 (2.35) & 3.4365e-03 (2.24) \\
 320 & 1e-01 & 4.4261e-11 & 1.6702e-03 (3.03) & 1.1434e-03 (3.01) \\
 \hline
 10 & 1e+00 & 3.6848e-15 & 1.0539e-01 (0.00) & 6.7925e-02 (0.00) \\
 20 & 1e+00 & 9.3571e-16 & 4.9678e-02 (2.12) & 3.0924e-02 (2.20) \\
 40 & 1e+00 & 3.8451e-16 & 2.3511e-02 (2.11) & 1.4248e-02 (2.17) \\
 80 & 1e+00 & 2.5252e-13 & 1.0821e-02 (2.17) & 6.3999e-03 (2.23) \\
 160 & 1e+00 & 3.9397e-12 & 4.5746e-03 (2.37) & 2.5857e-03 (2.48) \\
 320 & 1e+00 & 1.9792e-12 & 1.4744e-03 (3.10) & 7.3745e-04 (3.51) \\
 \hline
       \end{tabular}
\end{center}
  \caption{ $u^*$ is the  numerically obtained control (\ref{opt-control-num})  on
  a grid with $N=640$ mesh points after computation of $F(u^*)=0.$ with initial
  value $u(t)=1$ for the Gauss--Seidel iteration. $(u_n)_n$
 is  the numerically obtained control after succesfull computation of $F(u)=0$
 with the same initial value for the Gauss--Seidel iteration.
   The corrresponding
  states are $x^*$ and  $(x_n)_n$, respectively, and they are
 the discrete solution to (\ref{SRK-1}). Note that in order to
  compute  $F((u_n)_n)$  the adjoint states have obtained through (\ref{ARK-2})
  using IMEX SSP2. The value in the brackets is the residual of the current
  divided by the residual of the previous result. }
\label{res-SSP2}
\end{table}

\end{document}